\documentclass[12pt, reqno]{article}

\usepackage[usenames]{color}
\usepackage{amssymb}
\usepackage{amsmath}
\usepackage{amsthm}
\usepackage{amsfonts}
\usepackage{amscd}
\usepackage{graphicx}
\usepackage{hyperref}

\begin{document}

\begin{center}
\vskip 1cm{\LARGE\bf 
Ryser's Conjecture and Stochastic matrices}
\vskip 1cm
\large
Luis H. Gallardo \\
UMR CNRS 6205 \\
Laboratoire de Math\'ematiques de Bretagne Atlantique\\
6, Av. Le Gorgeu, C.S. 93837, Cedex 3, F-29238 Brest \\
France \\
\href{mailto:Luis.Gallardo@univ-brest.fr}{\tt Luis.Gallardo@univ-brest.fr} 
\end{center}

\vskip .2 in

\theoremstyle{plain}
\newtheorem{theorem}{Theorem}
\newtheorem{corollary}[theorem]{Corollary}
\newtheorem{lemma}[theorem]{Lemma}
\newtheorem{proposition}[theorem]{Proposition}

\theoremstyle{definition}
\newtheorem{definition}[theorem]{Definition}
\newtheorem{example}[theorem]{Example}
\newtheorem{conjecture}[theorem]{Conjecture}

\theoremstyle{remark}
\newtheorem{remark}[theorem]{Remark}
\newtheorem{notation}[theorem]{Notation}
\newtheorem{notations}[theorem]{Notations}
\newtheorem{condition}[theorem]{Condition}

\begin{abstract}
Assume that $H$ is a circulant Hadamard matrix of order $n\geq 4$. We consider an appropriate stochastic matrix $S$ of order $n$ depending on $H$. This allows us to prove that $n = 4$. Thus, there are only $10$ circulant Hadamard matrices.
\end{abstract}

\section{Introduction}

A matrix of order $n$ is a square matrix with $n$ rows.
A \emph{circulant} matrix $A := {\rm{circ}}(a_1,\ldots, a_n)$ of order $n$ is a  matrix of order $n$, with first row $[a_1, \ldots,a_{n}]$, in which each row after the first is
obtained by a cyclic shift of its predecessor by one position. For example, the second row of $A$
is $[a_{n}, a_1,\ldots,a_{n-1}]$. Three useful circulant matrices of order $n$ are the matrix $J$ with all its entries equal to $1$, i.e., $J := {\rm{circ}}(1,\ldots,1)$, the identity matrix $I := {\rm{circ}}(1,0,\ldots,0)$, and the zero matrix $0 :=  {\rm{circ}}(0,\ldots,0)$.
A \emph{Hadamard} matrix $H$ of order $n$ is a matrix of order $n$  with entries in $\{-1,1\}$  such that
$K := \frac{H}{\sqrt{n}}$ is an orthogonal matrix with rational entries. A \emph{circulant Hadamard} matrix of order $n$ is  a circulant matrix that is Hadamard of order $n$. The  $10$ known circulant Hadamard matrices are $H_1 :={\rm{circ}}(1), H_2 := -H_1, H_3 :={\rm{circ}}(1,-1,-1,-1), H_4 := -H_3, H_5 :={\rm{circ}}(-1,1,-1,-1), H_6 := -H_5, $ $H_7 := {\rm{circ}}(-1,-1,1,-1), H_8 := -H_7,
H_9 := {\rm{circ}}(-1,-1,-1,1), H_{10} := -H_9 $.
A \emph{stochastic} matrix $S$ of order $n$ is a matrix of order $n$ with non-negative entries such that the sum of all entries of each row of $S$ is equal to $1$. If the transpose of $S$ is also stochastic, then $S$ is a \emph{doubly stochastic} matrix. We denote by $A^{*}$ the conjugate transpose of the matrix $A$. When $A$ has real entries $A^{*}$ is simply the transpose of $A$.

Ryser proposed in $1963$ (see \cite[p. 97]{Davis}, \cite{Ryser}) the conjecture of the non-existence of circulant Hadamard matrices of order $>4$. Ryser's conjecture has since attracted some attention, see, e.g., \cite{brualdi,luisEJC,luisMC,Turyn}
and their bibliographies.

Brualdi \cite{brualdi}, a student of Ryser, proved the conjecture for every order $n$ provided $H$ is symmetric, and Turyn \cite{Turyn} proved the conjecture for all $n$'s of the form $n=4p^{2m}$ where $p$ is an odd prime number and $m$ is a positive integer.
Turyn \cite[pp. 329-330] {Turyn} also
proved the crucial fact that $n = 4 h^2$ with \emph{odd} $h$.

The object of the present paper is to prove the conjecture. It turns out that this follows immediately when we focus on the following  doubly stochastic matrix $S$ (see Lemma \ref{Sbisto}), associated to a  circulant Hadamard matrix $H$ of order $4 h^2$.
\begin{equation}
\label{ese}
S := \frac{H+J}{2(2 h^2+h)}.
\end{equation}

The link with the OEIS sequence A091704 (see \cite{oeis}) is the following. If Ryser's Conjecture
holds then there is no Barker codes of length $>13$ (see \cite{wikiB} ).

Our main result is the following.

\begin{theorem}
\label{ryserdone}
There is no circulant Hadamard matrices with more that $4$ rows.
\end{theorem}

 Section \ref{toools} contains some results useful for the proof of the theorem.
 The proof of the theorem appears in section \ref{prueba}.
 
\section{Tools}
\label{toools}
 The following is well known.  See, e.g., \cite[p. 1193]{HWallis}, \cite[p. 234]{Meisner}, \cite[pp. 329-330] {Turyn}.

\begin{lemma}
\label{regular}
Let $H$ be a regular Hadamard matrix of order $n\geq 4$, i.e., a Hadamard matrix whose row and column sums are all equal. Then $n=4h^2$ for some positive integer $h$.
Moreover, the row and column sums are all equal to $\pm2h$  and each row has $2h^2\pm h$ positive entries and $2h^2 \mp h$ negative entries. Finally, if $H$ is circulant then
$h$ is odd.
\end{lemma}

A simple computation follows.
 
\begin{lemma}
\label{Sbisto}
Let $H := {\rm{circ}}(h_1,\ldots,h_n)$ be a circulant Hadamard matrix of order $n$. Then, the matrix $S := \frac{H+J}{n+\sqrt{n}}$ is circulant and doubly stochastic.
\end{lemma}
 
 \begin{proof}
 By Lemma \ref{regular} $n = 4 h^2$. Sums and scalar multiples of circulant matrices are circulant. Thus, both $H+J$ and $S$ are circulant. Put $s$ the sum of all entries in any row of $S$. By Lemma \ref{regular} we can assume that the sum $r$ of all entries in any row of $H+J$ is equal to $2 h+4h^2$ (if not we change $H$ into $-H$).  Since $n+\sqrt{n} = 2(2 h^2+h)$ we obtain that $s = 1$. Thus $S$ is stochastic. Since $(H+J)^{*} = H^{*}+J^{*}= H^{*}+J $, and
\begin{equation}
\label{achestar}
 H^{*}+J = {\rm{circ}}(h_1+1,h_n+1,h_{n-1}+1,\ldots,h_2+1)\end{equation}
the sum of all entries in any row of $H^{*}+J$ is also equal to $r$. Thus $S^{*}$ is stochastic. This proves that $S$ is doubly stochastic.
\end{proof}
 
\section{Proof of Theorem \ref{ryserdone}}
\label{prueba}
Assume that $H$ is a circulant Hadamard matrix  with at least $4$ rows, i.e, $H$ has  order $n \geq 4$. By Lemma \ref{regular} $n= 4 h^2$ with odd $h$.
 We can assume  (if not we replace $H$ by $-H$) that there are $\frac{n+\sqrt{n}}{2}= 2 h^2+h$ entries equal to $1$ in row $1$ of $H$ (see Lemma \ref{regular}).  By Lemma \ref{Sbisto} the matrix $S$ defined in \eqref{ese} is circulant and doubly stochastic. It follows from \eqref{ese} that 
 \begin{equation}
\label{acheS}
H = 2 h(2 h +1) S - J.
\end{equation}
Computing the transpose conjugate in both sides of \eqref{acheS} we get 
\begin{equation}
\label{acheStar}
H^{*} = 2 h(2 h +1) S^{*} - J.
\end{equation} 
since $J$ is real and symmetric. But $H$ is Hadamard so that
\begin{equation}
\label{defH}
n I = H H^{*}.
\end{equation}
Using \eqref{defH}  together with \eqref{acheS} and \eqref{acheStar}; using that $I,J,H,H^{*},S,S^{*}$ are all circulant so that any product of two of them commute, and using that $J^2 = n J$ we obtain the following equality:
 \begin{equation}
\label{C1}
4 h^2 I = 4 h^2 (2 h +1)^2 S S^{*} - 2 h (2 h +1)(S + S^{*})J + 4 h^2 J.
\end{equation}
 But $S J = J$ and $S^{*} J = J$ since $S$ is doubly stochastic. Thus, dividing both sides of \eqref{C1} by $h$ we get
 \begin{equation}
\label{C2}
4 h I = 4 h (2 h +1)^2 S S^{*} - 4  (2h  + 1) J +4 h J.  
 \end{equation}
 Reducing both sides of \eqref{C2} modulo $h$ we obtain
 \begin{equation}
\label{C3}
0  = - 4 J \pmod{h}.  
\end{equation}
 Therefore, \eqref{C3} says that
  \begin{equation}
\label{C4}
h \mid 4.  
\end{equation}
 But $h$ is odd by Turyn's result in Lemma \ref{regular}. Thus, \eqref{C4} implies that
   \begin{equation}
\label{C5}
h =1.  
\end{equation}
In other words, we have proved that $n=4$. This proves the theorem.

\section{Acknowledgments}
I am grateful to Jeffrey Shallit for suggesting me possible mathematical journals where to submit the present paper.

\end{document}